\documentclass[10pt]{article}
\usepackage{amsmath,amssymb,amsthm,amscd}
\numberwithin{equation}{section}

\def\cS{{\mathcal S}}

\def\RR{{\mathbb R}}
\def\CC{{\mathbb C}}

\newtheorem{prop}{Proposition}[section]
\newtheorem{theo}[prop]{Theorem}
\newtheorem{lemm}[prop]{Lemma}
\newtheorem{coro}[prop]{Corollary}

\newtheorem{conj}[prop]{Conjecture}
\def\begeq{\begin{equation}}
\def\endeq{\end{equation}}

\def\lab{\ }
\def\lab{\label}

\title{A continuity method to construct canonical metrics}

\author{Gabriele La Nave\\University of Illinois\\
\\Gang Tian\thanks{Supported partially by
grants from NSF and NSFC}
\\Beijing University and Princeton University}
\date{}
\begin{document}
\maketitle
\tableofcontents

\section{Introduction}

The K\"ahler-Ricci flow has played a fundamental role in the Analytic Minimal Model Program.
There has been quite a bit of progresses and many very important results have been proven (e.g., \cite{tianzhang06}, \cite{songtian07}, \cite{songtian09} et al). In this short paper,
we introduce a new continuity method
which provides an alternative way of carrying out the Analytic Minimal Model Program. This method may not be as natural as the K\"ahler-Ricci flow, but it
has the advantage of having Ricci curvature bounded from below along the deformation, so many existing analytic tools, such as the compactness theory of Cheeger-Colding-Tian for K\"ahler manifolds and the partial $C^0$-estimate, can be applied.

Assume that $M$ is a compact K\"ahler manifold with a K\"ahler metric $\omega_0$. We consider the ~1-parameter family of equations:
\begin{equation}\label{eq:main-1}
\omega = \omega _0 - t\,{\rm Ric }(\omega).
\end{equation}
Clearly, the K\"ahler classes vary according to the linear relation:
$[\omega]\,=\,[\omega_0] - t\,c_1(M)$, where $[\omega] \in H^2(M,\RR) \cap H^{1,1}(M)$
denotes the K\"ahler class of $\omega$.

Our first theorem is:
\begin{theo}\label{th:main-1}
For any initial K\"ahler metric $\omega_0$, there is a smooth family of solutions $\omega_t$ for \eqref{eq:main-1} on $M\times [0,T)$, where
\begin{equation}\label{eq:T}
T\,=\,\sup \{\,t\,|\, [\omega_0]\,-\,t\,c_1(M)\,>\,0\,\}.
\end{equation}
\end{theo}
This is an analogue of the sharp local existence theorem for the K\"ahler-Ricci flow due to Z. Zhang and the second author \cite{tianzhang06} and it is not hard to prove.

If $T < \infty$, we need to examine the limit of $\omega_t$ as $t$ tends to $T$. In general, this is highly non-trivial. However, we can still
prove the following:

\begin{theo}
\label{th:main-2}
Assume that $T< \infty$ and $([\omega_0]-T c_1(M))^n > 0$, where $n=\dim_\CC M$,
then $\omega_t$ converge to a unique weakly K\"ahler metric $\omega_T$ such that $\omega_T$ is smooth on $M\backslash \cS$, where $\cS$ is a subvariety, and satisfies:
\begin{equation}\label{eq:T-2}
\omega_T\,=\,\omega_0\,-\,T\,{\rm Ric}(\omega_T)~~~{\rm on}~~M\backslash \cS.
\end{equation}
Furthermore, $\cS$ is the base locus of $[\omega_0] - T c_1(M)$, i.e., the set of points where $[\omega_0] - T c_1(M)$ fails to be positive.
\end{theo}

For any $\omega_t$ above ($t<T$), we have
$${\rm Ric}(\omega_t) \,=\, t^{-1}\,(\omega_0\,-\,\omega_t)\,\ge\, t^{-1}\,\omega_t.$$
In particular, the Ricci curvature of $\omega_t$ is bounded from below near $T$. We do expect a uniform bound on the diameter of
$\omega_t$ for any $t\in (0,T)$ even if $T=\infty$. If so, by taking subsequences if necessary, we may assume that
$(M,\omega_t)$ converge to a length space $(M_T, d_T)$ in the Gromov-Hausdorff topology. Then one should be able to further prove that
$(M_T, d_T)$ is the metric completion of $(M\backslash \cS, \omega_T)$.

If $T=\infty$, we expect that $\omega_t$, after appropriate scaling, converge to a weakly K\"ahler-Einstein metric or a generalized K\"ahler-Einstein metric --
introduced by Song and the second author in \cite{songtian09}-- on the canonical model of $M$. In fact, if $(-c_1(M))^n > 0$, we can verify this by using the same arguments
in the proof of Theorem \ref{th:main-2} (see Section 3).

The organization of this paper is as follows: In the next section, we prove Theorem \ref{th:main-1} by using standard arguments for complex Monge-Ampere equations.
In Section 3, we prove Theorem \ref{th:main-2}. In Section 4, we describe the Analytic Minimal Model Program by using this new continuity method. This is parallel to what has been done for the K\"ahler-Ricci flow (see \cite{tian02}, \cite{tian07}, \cite{songtian07} and \cite{songtian09}). We will also propose
a number of problems for carrying out the program.

\section{Maximal solution time}

In this section, we prove Theorem \ref{th:main-1}. First we reduce \eqref{eq:main-1} to a scalar equation. Choose
a real closed $(1,1)$ form
$\psi$ representing $ c_1(X)$ and a smooth volume form $\Omega$ such that ${\rm Ric}({\Omega})=
\psi$. This $\Omega$ is unique up to multiplication by a positive constant.

Set $\tilde \omega_{t}= \omega_{0}-t \psi$ for $t\in [0,T)$. One can easily show that
$\omega =\tilde \omega_t + t \sqrt{-1}\partial \bar{\partial} u$ satisfies \eqref{eq:main-1} if $u$ satisfies
\begin{equation}
\lab{eq:p-1}
(\tilde\omega_t\,+\,t \sqrt{-1}\,\partial\bar\partial \,u)^n\,=\,e^{ u}\, \Omega,
\end{equation}
where $\tilde\omega_t\,+\,t\,\partial\bar\partial \,u\,>\,0$. This equation depends on the choice of $\psi$, but if we choose a different
representative $\hat \psi$ of $c_1(M)$, \eqref{eq:p-1} changes in a simple way: Write $\hat \psi = \psi + \sqrt{-1}\,\partial\bar\partial\,v$
and $\hat\omega_t = \omega_0 - t \hat \psi$, then for any solution $u$ of \eqref{eq:p-1}, $\hat u = u-v$ solves
\begin{equation}
\lab{eq:p-1'}
(\hat \omega_t \,+\,t \sqrt{-1}\,\partial\bar\partial \,\hat u)^n\,=\,e^{ \hat u}\, \hat \Omega,
\end{equation}
where $\hat \Omega$ is a volume form satisfying:
$${\rm Ric}({\hat \Omega})= \hat \psi,~~~\int_M\, e^u\,\Omega \,=\,\int_M\, e^{\hat u}\,\hat \Omega.$$
This shows that the solvability of \eqref{eq:p-1} is independent of the choice of $\psi$.

On the other hand, it follows from the definition of $T$ that for any $\bar t < T$, there is a $\psi$ such that $\omega_0 - \bar t \,\psi > 0$.
Thus, in order to prove Theorem \ref{th:main-1},
we only need to prove that \eqref{eq:p-1} is solvable for such a $\psi$ and any $t\in [0,\bar t]$. Put
$$E\,=\,\{ t\in [0,\bar t]\,|\,\eqref{eq:p-1} ~{\rm has~a~solution}~\}.$$
Clearly, $0\in E$ since $u=\log \left (\frac{\omega^n_0}{\Omega}\right)$ is an obvious solution. So $E$ is non-empty.

\begin{lemm}\label{lemm:2-1}
The set $E$ is open.
\end{lemm}
\begin{proof}
Assume that $t_1\in E$ and $u_1$ be a solution of \eqref{eq:p-1} with $t=t_1$. We want to solve \eqref{eq:p-1} for $t$ close to $t_1$.
If $t_1 > 0$, that is readily done Write $ t u = t_1 u _1 + w$ for some small $w$. Then \eqref{eq:p-1} becomes
$$  (\omega_1 \,-(t-t_1) \psi+ \,\sqrt{-1}\,\partial\bar\partial \,w)^n\,=\,e^{\frac{w}{t} - \frac{u_1\,(t-t_1) }{ t}}\,\omega_1^n,$$
where $\omega_1 = \omega_0-t_1\psi \,+\,t_1 \sqrt{-1}\,\partial \bar\partial \,u_1$.  
%%%
\noindent
In fact, $\Omega = e^{-u_1} \omega _1^n$ and $\omega _t= \omega _0 - t_1 \psi - (t-t_1)\psi + \sqrt{-1} \partial \bar \partial (t_1u_1+w)= \omega_1 \,-(t-t_1) \psi+ \,\sqrt{-1}\,\partial\bar\partial \,w$ and finally we have used that $ e^u\Omega= e^{\frac{t_1u_1+ w}{t}- u_1}\omega _1^n$.
Naturally, one still has that $ \omega_1 \,-(t-t_1) \psi>0 \text{ for } t-t_1 \text{ sufficiently small }$ and therefore, setting $\hat \omega _1 :=  \omega_1 \,-(t-t_1) \psi>0 $ 
 and choosing $F \text{ such that }\hat  \omega _1^n = e^F \omega _1^n$, one can write the equation above as:
   $$   (\hat \omega_1 \,+ \,\sqrt{-1}\,\partial\bar\partial \,w)^n\,=\,e^{\frac{w}{t} - \frac{u_1\,(t-t_1) }{ t}-F}\,\hat \omega_1^n,$$
 From now on, for the sake of notation we shall write $\omega_1$ instead of $\hat \omega _1$.

Expanding in $w$, we get
$$\Delta_1 w - t^{-1} w\,=\, \frac{u_1\,(t_1 -t)}{t}\,+\,Q(\nabla^2 w),$$
where $\Delta_1$ denotes the Laplacian of $\omega_1$ and $Q(a)$ denotes a polynomial in $a$ starting with a quadratic term.
Applying the Implicit Function Theorem, one can easily solve the above equation for $w$, consequently $u$,  for $t-t_1$ sufficiently small.

If $t_1=0$, we need to work a bit more carefully since the left side of
\eqref{eq:p-1} is degenerate. Write
$$ u \,=\, \log \left (\frac{\omega^n_0}{\Omega}\right)\,+\,t^{-1} w.$$
Then \eqref{eq:p-1} becomes
\begin{equation}
\label{eq:p-1-1}
\Delta_0 w \,-\,t^{-1} w \,=\,-\,t \,\Delta_0 \log \left (\frac{\omega^n_0}{\Omega}\right) \,+\,Q\left (t \nabla \log \left (\frac{\omega^n_0}{\Omega}\right) \,+\,\nabla ^2 w\right ),
\end{equation}
where $\Delta_0$ is the Laplacian of $\omega_0$. Put
$$A\,=\,\left |\left | \Delta_0 \log \left (\frac{\omega^n_0}{\Omega}\right)\right |\right |_{C^{\frac{1}{3}}}.$$

\vskip 0.1in
\noindent
{\bf Claim}: There is a uniform constant $C$ such that for any $f\in C^{\frac{1}{3}}$, there is a solution $v$ satisfying:
$$\Delta_0 v - t^{-1} v = t f~~~{\rm and}~~~t ^{-1} ||v||_{C^0} + || v||_{C^{2,\frac{1}{3}}} \le C  t^{\frac{2}{3}}\,||f||_{C^{\frac{1}{3}}}.$$
\begin{proof} Let us prove this claim. First by the Maximum Principle, $ |v| \le A t^2$, with $A:=\Vert f\Vert _{\infty}$. Then by standard elliptic theory, $||v||_{C^1} \le C' t$.
We can deduce from these: $t^{-1} ||v||_{C^{\frac{1}{3}}} \le C'' t^{\frac{2}{3}}$.\footnote{Both $C'$ and $ C''$ are uniform constants. }.
Then by elliptic theory again, we get $|| v||_{C^{2,\frac{1}{3}}} \le C '''A t^{\frac{2}{3}}$. The claim is proved.
\end{proof}
Now we can complete the proof of this lemma by standard iteration: Set $w_0=0$ and construct $w_i$ for $i\ge 1$ by solving the equation:
\begin{equation}
\label{eq:iterate}
\Delta_0 w_i \,-\,t^{-1} w_i \,=\,-\,t \,\Delta_0 \log \left (\frac{\omega^n_0}{\Omega}\right) \,+\,Q\left (t \nabla \log \left (\frac{\omega^n_0}{\Omega}\right) \,+\,\nabla ^2 w_{i-1}\right ).
\end{equation}
If $||w_{i-1}||_{C^{2,\frac{1}{3}}} \le C (1+A) t^{\frac{2}{3}}$, then for $t$ sufficiently small, the right side of \eqref{eq:iterate} is bounded
by $A\, t$, so
by the above claim, we get
$$||w_i||_{C^{2,\frac{1}{3}}}\, \le\, C\, (1+A)\, t^{\frac{2}{3}}.$$
Moreover, by the above claim,
we have
$$t^{-\frac{2}{3}}\,||w_{i+1}\,-\,w_{i}||_{C^{2,\frac{1}{3}}} \,\le\, C\,t^{\frac{1}{3}}\, \left(t^{-\frac{2}{3}}\,||w_i\,-\,w_{i-1}||_{C^{2,\frac{1}{3}}}\right ). $$
It follows that for sufficiently small $t$,
$w_i$ converge to a $C^2$-function $w$ which solves \eqref{eq:p-1-1}, so we get a solution for \eqref{eq:p-1} for $t$ sufficiently small, i.e., $E$ is open.

\end{proof}

It remains to prove that $E$ is closed.
Assume that $\{t_i\}\subset E$ is a sequence with
$\lim t_i = \bar t > 0$ and $u_i$ is the solution of \eqref{eq:p-1} with $t=t_i$. We want to prove $\bar t\in E$.
It amounts to getting an a priori $C^{2,1/2}$-estimate for $u_i$. By applying the Maximum Principle to \eqref{eq:p-1}, we have
$$ \sup_M \,|u_i| \,\le\, \sup_M\,\left | \,\log \left (\frac{\omega^n_0}{\Omega}\right) \right |.$$
So we have the $C^0$-estimate. For the $C^2$-estimate, it is easier to use the Schwartz-type estimate.

Using \eqref{eq:p-1} or equivalently, \eqref{eq:main-1}, we have
$${\rm Ric}(\omega_i)\, =\,\frac{1}{t_i}\, (\omega_0\,-\,\omega_i)\,\ge\,-\,\frac{1}{t_i}\,\omega_i,$$
where $\omega_i = \tilde \omega_{t_i}\,+\,t_i\,\sqrt{-1}\,\partial\bar\partial \,u_i$. then by standard computations, we have
$$\Delta_i \log {\rm tr}_{\omega_i}(\tilde \omega_{t_i}) \,\ge \, - a \, {\rm tr}_{\omega_i}(\tilde\omega_{t_i})\,-\,\frac{1}{t_i},$$
where $\Delta_i$ is the Laplacian of $\omega_i$ and $a_i$ is a positive upper bound of the bisectional curvature of $\tilde \omega_{t_i}$.
Since $\lim t_i =\bar t$ and $\{\tilde\omega_t\}$ is a smooth family of K\"ahler metrics for $t$ near $\bar t$, we have $a= \sup_i a_i < \infty$.

If we put $$v \,=\,\log {\rm tr}_{\omega_i}(\tilde \omega_{t_i})\, - \,(a+1)\, t_i \,u_i,$$
then it follows
$$ \Delta_i v \,\ge \, e^{v - (a +1)\,t_i\, c}\, - \,n (a+1)\,-\,\frac{1}{t_i},$$
where $c= \sup _i (- \inf_M u_i)$. Hence, by using the Maximum Principle, we can bound $v$ from above, so there is a uniform constant $C$ such that
$$ C^{-1}\,\omega_0\,\le\,\omega_i.$$
Using \eqref{eq:p-1}, we derive
\begin{equation}\label{eq:lap-est}
 C^{-1} \,\tilde \omega_{t_i}\,\le \, \omega_i\,\le\,C\,\tilde\omega_{t_i}.
 \end{equation}
Next, by applying Calabi's 3rd derivative estimate, (cf. \cite{yau}) \footnote{One may also use the $C^{2,\alpha}$-estimate (cf. \cite{evans}, \cite{tian84}).}
for a uniform constant $C'$, we have
$$||u_i||_{C^3}\,\le\,C'.$$
Thus, by taking a subsequence if necessary, $u_i$ converge to a $C^3$-solution of \eqref{eq:p-1} with $=\bar t$, i.e., $\bar t\in E$ and $E$
is closed. Theorem \ref{th:main-1} is proved.

\begin{coro}\label{coro:long-1}
If $K_M$ is nef, then \eqref{eq:p-1} has a unique solution for any $t > 0$.
\end{coro}

\section{Proof of Theorem \ref{th:main-2}}
First we observe
\begin{lemm}
\label{lemm:3-1}
For any solution $u$ of \eqref{eq:p-1} with $t \in (0,T)$, we have
$$\sup_M \,u \,\le\, \sup_M \,\log \left (\frac{\omega^n_t}{\Omega}\right),$$
i.e., $u$ is uniformly bounded from above.
\end{lemm}

The proof of Theorem \ref{th:main-2} is now quite standard. For simplicity, we assume that $M$ is a projective manifold and $[\omega_0] = 2\pi c_1(L)$
for some line bundle $L$. Under this assumption, we can use the following lemma due to Kodaira (cf. \cite{kawamata84}). The general case can be proved in an identical way
if we use instead Demailly-Paun's extension \cite{demaillypaun} of Kodaira's lemma to K\"ahler manifolds.

\begin{lemm}
\label{lemm:2-2}
Let $E$ be a divisor in a projective manifold $M$. If $E$ is
nef and big, then there is an effective $\RR$ divisor $D = \sum_i \,a_i D_i$, where $D_i$ are divisors and
$a_i$ are positive real numbers, such that $E - [D] > 0$,
where $[D]=\sum _i a_i [D_i]$ and each $[D_i]$ denotes the line bundle induced by $D_i$.
\end{lemm}

This follows from the openness of the big cone
of $M$ which clearly contains the positive cone and the fact that $E$ is in the closure
of the positive cone. Recall that the non-ample locus $B_+(E)$ of $E$ is defined to be the intersection of
$Supp(D)$, where $D$ ranges over all effective $\RR$-divisor such that $E - D$ is ample.
The above lemma implies that $B_+(E)$ is a subvariety whenever $E$ is big.

Now we take $E= L+  T K_M$ which is nef and big, so by the above lemma, for any
$x\in M\backslash B_+(E)$, there is an effective divisor $D = \sum_i \,a_i D_i $ as above
such that $L + T K_M - [D]$ is ample.
%so there is a K\"ahler metric $\omega_D$ with K\"ahler class $[\omega_0 - 2\pi T( c_1(M) + [D])$.

For each $i$, let $\sigma_i$ be the defining section of $D_i$ and $||\cdot||_i$ be a Hermitian norm on $[D_i]$.
Then $\sum_i a_i \log ||\sigma_i||_i^2$ is a well-defined function outside $Supp (D)\subset M$.
Since $L +  T K_X - [D]$ is ample, we can choose $||\cdot||_i$ such that
$$\tilde \omega_T\, +\, \sqrt{-1}\, \sum_i \,a_i \partial\bar\partial
\log ||\sigma_i||_i^2 > 0.$$
For simplicity, we will write
$$\log ||\sigma||^2 = \sum_i a_i \log ||\sigma_i||_i^2.$$
Formally, we regard $\sigma $ as a defining section of $D$ and $||\cdot ||$ as an norm on $[D]$.

Set
$$\tilde \omega_{t,D}\,=\, \tilde\omega_t\, +\, \sqrt{-1}\,\partial\bar\partial\, \log ||\sigma||^2.$$
Then there is a $\delta=\delta(D)$, which may depend on $D$, such that $\omega_{t,D}$ is a
smooth K\"ahler metric for any $t\in [T-\delta, T+ \delta]$. For any solution $u$ of \eqref{eq:p-1} with $t\in [T-\delta, T]$, we put
$$v\, =\,  t\, u\, - \,\log ||\sigma||^2,$$
then $\omega_t = \tilde \omega_{t,D} + \sqrt{-1}\,\partial\bar\partial\, v$ and $v$ satisfies the following equation:
\begin{equation}
\lab{eq:p-3}
(\tilde\omega_{t,D}\,+\, \sqrt{-1}\,\partial\bar\partial \,v)^n\,=\,e^{\frac{1}{t} (v + \log ||\sigma||^2)}\, \Omega.
\end{equation}
Applying the Maximum Principle to \eqref{eq:p-3}, we get $ v\,\ge\, -c $ for a uniform constant which may depend on $D$.
So by changing $c$ if necessary, we have
$$-c \,\le\, v\,\le\, c - \log ||\sigma||^2.$$
In particular, $v$ or equivalently $u$, is bounded outside $D$. Since ${\rm Ric}(\omega_t)$ is bounded below by $-1/t$, as in last section, we can infer
$$\Delta_t \log {\rm tr}_{\omega_t}(\tilde \omega_{t,D}) \,\ge \, - a \, {\rm tr}_{\omega_t}(\tilde \omega_{t,D})\,-\,\frac{1}{t},$$
where $\Delta_t$ denotes the Laplacian of $\omega_t$ and $a$ is a positive upper bound on the bisectional curvature of $\tilde \omega_{t,D}$
for all $t\in [T-\delta,T]$.
Put
$$w \,=\,\log {\rm tr}_{\omega_t}(\tilde \omega_{t,D})\, - \,(a+1) \,v.$$
Then we have
$$ \Delta_t w \,\ge \, e^{w - (a+1)\, b}\, - \,n (a+1)\,-\,\frac{1}{t},$$
where $b = - \inf_M \,v$.
Hence, by using the Maximum Principle, we can bound $w$ from above, so there is a uniform constant $C$ such that
$$ C^{-1}\,||\sigma||^{2(a+1)}\,\tilde \omega_{t,D}\,\le\,\omega_t.$$
Using \eqref{eq:p-3}, we derive
\begin{equation}\label{eq:lap-est-2}
 C^{-1} \,||\sigma||^{2(a+1)}\,\tilde\omega_{t,D}\,\le \, \omega_t\,\le\,C\,||\sigma||^{-2 (n-1)(a+1) + \frac{1}{t}}\,\tilde \omega_{t,D}.
\end{equation}
Then by Calabi's 3rd derivative estimate, for any compact subset $K\subset M\backslash D$, we have $||u||_{C^3(K)}\,\le\,C_K$ for some
uniform constant $C_K$. It follows that any sequence $\{t_i\}$ with $t_i\to T$ has a subsequence, still denoted by $t_i$ for simplicity, such that
$v_{t_i}$ converge a $C^3$-function $v_t$ on $M\backslash D$ satisfying:
\begin{equation}
\lab{eq:p-3'}
(\tilde\omega_{T,D}\,+\, \sqrt{-1}\,\partial\bar\partial \,v_T)^n\,=\,e^{\frac{1}{T} (v_T + \log ||\sigma||^2)}\, \Omega.
\end{equation}
A priori, this limit may not be unique. so we still need to prove that $v_T$ is unique, i.e., independent of the sequence $\{t_i\}$.
\begin{lemm}
\label{lemm:2-3}
Let $\dot v$ be the derivative of $v$ in the $t$-direction and $T < \infty$. Then there is a uniform constant $C_T$, which may depend on $T$, such that
\begin{equation}\label{eq:t-derivative}
\int_M \,|\dot v|^2 \,\omega_t^n\,\le\, C_T ~~~{\rm on}~~t\in (T-\delta, T).
\end{equation}
\end{lemm}
\begin{proof}
Differentiating \eqref{eq:p-3} on $t$, we get
$$\Delta_t \dot v\,=\,-\frac{1}{t^2} \,u \,+\,\frac{1}{t}\,\dot v.$$
Note that $\dot v$ is equal to $t \dot u + u $ which is smooth on $M\times (0, T)$
since $\log ||\sigma||^2$ is independent of $t$.
Since $u$ is bounded from above, we deduce from \eqref{eq:p-1}
$$\int_M \,|u|^2\, \omega_t^n\,=\,\int_M\,u^2\, e^u\,\Omega\,\le\,C,$$
where $C$ is a uniform constant. Then we have
\begin{equation}\label{eq:t-derivative2}
\int_M \,\left (|\nabla \dot v|^2\,+\,\frac{1}{t}\,|\dot v|^2\right )  \,\omega_t^n\,=\,\frac{1}{t^2}\,\int_M\, {\dot v}\,u\,\omega_t^n.
\end{equation}
Then \eqref{eq:t-derivative} follows easily from this and the Cauchy inequality.
\end{proof}
For any compact subset $K\subset M\backslash D$, by \eqref{eq:lap-est-2}, we have that for some $C_K > 0$,
$$C_K^{-1}\,\omega_0\,\le\omega_t\,\le\,C_K\,\omega_0.$$
Therefore, we have for $T-\delta < t <t'<T$,
$$\int_K\,|v(x,t) - v(x,t')|^2\,\omega_0^n(x)\,\le\,\int_t^{t'} \int_K\,|\dot v (\cdot,s)|^2\,\omega_0^n \,ds\,\le\,C_T (C_K)^n\,|t'-t|.$$
It follows that $v(\cdot,t)$ and $v(\cdot,t')$ converge to the same function $v_T$ on $M\backslash D$.

Since $D$ is any divisor in the definition of $B_+(E)$, where $E\,=\,L + T\,K_M$, we have proved Theorem \ref{th:main-2} with ${\cal S}\,=\,B_+(E)$.

\section{Analytic Minimal Model Program revisited}

In this section, we follow the lines of the approach towards the Analytic Minimal Model Program through Ricci flow (cf.
\cite{tian07}, \cite{songtian07}, \cite{songtian09} et al)
to list some problems and speculations.
Some of these problems are doable by adapting arguments from what has been done for the K\"ahler-Ricci flow (cf. above citations). It is possible to get even stronger results because
Ricci curvature is bounded from below in this new approach through the continuity method.

Let $u_t$ be the maximal solution of \eqref{eq:p-1} for $t\in(0,T)$ and write 
$$\omega_t=\omega_0-t\psi+\sqrt{-1}\,\partial\bar\partial u_t.$$
First we assume $T < \infty$.

\begin{conj}
\label{conj-1}
As $t\to T$, $(M, \omega_t)$ converges to a compact metric space $(M_T, d_T)$ in the Gromov-Hausdorff topology
\footnote{Since $\omega_t$ has Ricci curvature bounded from below, this is equivalent to that the diameter of $(M,\omega_t)$
is uniformly bounded.} satisfying the following:

\vskip 0.1in
\noindent
(1) $M_T$ is a K\"ahler variety and there is a holomorphic fibration $\pi_T: M\mapsto M_T$;

\vskip 0.1in
\noindent
(2) $d_T$ is a ``nice'' K\"ahler metric $\omega_T$ on $M_T\backslash {\cal S_T}$, where ${\cal S}_T$ is a subvariety of $M_T$ containing all the singular points.
If $([\omega_0] - T\,c_1(M))^n >0$, then it is the same as saying that $\omega_T$ is the limit given in Theorem \ref{th:main-2}
and $(M_T, d_T)$ is the metric completion of $(M\backslash {\cal S}, \omega_T)$;

\vskip 0.1in
\noindent
(3) $\omega_t$ converge to $\omega_T$ on $\pi^{-1}(M_T\backslash {\cal S}_T)$ in a much regular topology, possibly, the smooth topology.
\end{conj}
It seems that the key for solving this conjecture is to bound the diameter of $(M,\omega_t)$, especially, in the non-collapsing case.

Next we want to examine how to extend our continuity method beyond $T$. To this end, we propose:

\begin{conj}
\label{conj-2}
The variety $M_T$ has a partial resolution $\pi'_T: M'\mapsto M_T$, which may have ``mild'' singularity, such that the canonical sheaf
$K_{M'}$ is well-defined and $\pi'^\*_T(\pi_{T\*} [\omega_T]) + t K_{M'} > 0$ for  $t>0$ small.
If $\dim M_T =\dim M$, i.e., in the non-collapsing case, $M'$ should be a flip of $M$ as defined in algebraic geometry.
\end{conj}
This may be the most difficult part of the program as we have learned from the Analytic Minimal Model Program through Ricci flow.
More precisely, by ``mild'' singularity, we mean the following
\begin{conj}
\label{conj-3}
Let $\omega_0' =  \pi'^*_T(\pi_{T*} \omega_T)$. Then we can solve \eqref{eq:p-1} on $M'$ for $t\in (0, T')$
with initial metric $\omega_0'$, where
$$T'\,=\,\sup \{t\,|\, \pi'^*_T(\pi_{T*} [\omega_T]) + t K_{M'} > 0\,\} .$$
\end{conj}
By the work in \cite{songtian09}, one knows that log-terminal singularities have the property in the  conjecture above, for the Ricci flow.
The same should be true for the new continuity method. In view of \cite{songtian09}, it may be possible to solve Conjecture \ref{conj-3} by the current technology at hand.

If the above conjectures can be affirmed, for some initial K\"ahler metric $\omega_0$ on $M$,
we can construct a family of  pairs $(M_t, \omega_t)$ ($0\le t < \infty$) together with
a sequence of times $T_0=0 < T_1 < T-2 < \cdots < T_k < \cdots$, satisfying:

\vskip 0.1in
\noindent
(1) For $t\in [T_{i}, T_{i+1})$ ($i\ge 0$), $M_t = M_i$ is a fixed K\"ahler variety \footnote{$M_i$ may be of smaller dimension or even an empty set.}
and $\omega_t$ is a solution
of \eqref{eq:main-1} on $M_i$ in a suitable sense. Moreover, $M_0=M$;

\vskip 0.1in
\noindent
(2) For $i \ge 1$, $M_{i+1}$ is a ``flip'' of $M_{i}$ as described in Conjecture \ref{conj-2}. If we denote by $\pi_{i}: M_{i} \mapsto M_{T_i}$ and
and $\pi_{i+1}: M_{i+1}\mapsto M_{T_i}$ the natural projections in the flip process, then
$$\lim _{t\to T_i - }\,\pi_{i*} \omega_t \,= \,\lim_{t\to T_i+}\,\pi_{i+1*} \omega_t.$$
We also have $\lim_{t\to 0}\,\omega_t = \omega_0$;

\vskip 0.1in
\noindent
(3) $\omega_t$ are smooth on the regular part of $M_t$ and continuous on the level of potentials along $t$.
\vskip 0.1in
Such a family corresponds to the solution of the K\"ahler-Ricci flow with surgery.

As in the case of K\"ahler-Ricci flow, we call $T_i$ a surgery time.
We expect that for each initial K\"ahler metric $\omega_0$, there are only finitely many surgery times, that is,

\begin{conj}
\label{conj-4}
There are only finitely many surgery times $T_0=0 < T_1 < T_2 <\cdots < T_N < \infty$ such that $M_t= M_N$ for $t > T_N$ is either empty or
a minimal model of $M$. If $M_N\not= \emptyset$, then $K_{M_N}$ is nef., consequently, \eqref{eq:main-1} admits a solution $\omega_t $ for all $t > T_N$.
If $M_N = \emptyset$, $M$ is birational to a Fano-like manifold and the converse os also true.
\end{conj}

Next, assuming that $M$ is a minimal model, smooth or with ``mild'' singularities, we need to analyze the asymptotic of
solutions $\omega_t$ of \eqref{eq:main-1} as $t$ tends to $\infty$.

\begin{conj}\label{conj-5}
If the Kodaira dimension $\kappa(M)=0$, then $\omega_t$ should converge to a Calabi-Yau metric $\omega_\infty$ on $M\backslash {\cal D}$, where $D$ is a subvariety, in the smooth topology. Furthermore,
$(M, \omega_t)$ should converge to a compact metric space $(M_\infty, d_\infty)$ in the Gromov-Hausdorff topology such that $M_\infty$ is the metric completion of
$(M\backslash D, \omega_\infty)$ and $d_\infty|_{M\backslash D} $ is induced by $\omega_\infty$.
\end{conj}

Similarly, we expect
\begin{conj}\label{conj-6}
If $\kappa(M)=\dim M = n$, then $t^{-1} \omega_t$ should converge to
a K\"ahler-Einstein metric $\omega_\infty$ with scalar curvature $-n$ on $M\backslash {\cal D}$, where $D$ is a subvariety, in the smooth topology. Furthermore,
$(M, t^{-1} \omega_t)$ should converge to a compact metric space $(M_\infty, d_\infty)$ in the Gromov-Hausdorff topology such that $M_\infty$ is the metric completion of $(M\backslash D, \omega_\infty)$ and $d_\infty|_{M\backslash D} $ is induced by $\omega_\infty$.
\end{conj}
In view of recent works of Jian Song \cite{song14}, we believe that the above two conjectures are solvable. The key should be a
diameter estimate on $(M, \omega_t)$.

It remains to consider the cases: $1\le \kappa(M) \le n-1$. In these cases, we can not expect the existence of any K\"ahler-Einstein metrics (even with possibly singular along a subvariety) on $M$ since $K_M^n=0$. However, we expect
\begin{conj}\label{conj-7}
If $1\le \kappa(M)\le n$, then $(M,t^{-1} \omega_t)$ should converge to a compact metric space $(M_\infty, d_\infty)$ in the Gromov-Hausdorff topology such that $M_\infty$ is a K\"ahler variety of complex dimension $\kappa(M)$ and $d_\infty$ is induced by a generalized K\"ahler-Einstein metric $\omega_\infty$
(cf. \cite{songtian09}) on the regular part of
$M_\infty$. Moreover, $(M, t^{-1} \omega_t)$ should converge to $\omega_\infty$ in a stronger topology, such as in $C^{1,1}$-topology on K\"ahler potentials,
on the regular part of $M_\infty$.
\end{conj}
The existence of generalized K\"ahler-Einstein metrics was established in \cite{songtian09}.

\end{document}